\documentclass[twoside,12pt]{article}
\usepackage{amssymb}
\usepackage{amsfonts}
\usepackage{amscd}
\usepackage{amsmath}
\usepackage{amsthm}

\advance\topmargin by -\headheight
\advance\topmargin by -\headsep

\newtheorem{theo}{Theorem}
\newtheorem{prop}[theo]{Proposition}
\newtheorem{coro}[theo]{Corollary}
\newtheorem{lemm}[theo]{Lemma}

\theoremstyle{remark}
  \newtheorem*{rema*}{Remark}
  \newtheorem{rema}[theo]{Remark}
    
  \newtheorem*{defi*}{Definition}  
    
\newcommand{\comment}[1]{}

\newcommand \al{\alpha}
\newcommand\be{\beta}

\newcommand\et{\eta}

\newcommand\si{\sigma}

\newcommand\ph{\varphi}

\newcommand\ps{\psi}
\newcommand\om{\omega}
\newcommand\Ga{\Gamma}

\newcommand\Th{\Theta}
\newcommand\La{\Lambda}

\newcommand\Om{\Omega}

\newcommand\ie{i.e.\ }

\def\RR{\mathbb R}

\renewcommand\o{\circ}

\newcommand\x{\times}
\newcommand\on{\operatorname}

\newcommand\pr{\on{pr}}

\newcommand\ev{\on{ev}}

\newcommand\dd{\mathbf{d}}

\newcommand\Diff{\on{Diff}}

\newcommand\Emb{\on{Emb}}

\newcommand\Gr{\on{Gr}}

\newcommand\pa{{\on{\partial}}}

\newcommand\g{\mathfrak g}

\newcommand\X{\mathfrak X}

\newcommand\M{{\mathcal M}}

\newcommand\F{\mathcal{F}}

\newcommand{\fint}{-\!\!\!\!\!\!\int}

\def\XXint#1#2#3{{\setbox0=\hbox{$#1{#2#3}{\int}$ }
\vcenter{\hbox{$#2#3$ }}\kern-.5\wd0}}

\date{ }
\begin{document}

\title{Induced differential forms on\\ manifolds of functions}
\author{Cornelia Vizman \\\it \small West University of Timi\c soara,
Department of Mathematics\\ 
%\it\small Bd. V.Parvan 4, 300223--Timi\c soara, Romania\\
\it \small e-mail: vizman@math.uvt.ro}
\maketitle

\begin{abstract}
Differential forms on the Fr\'echet manifold $\F(S,M)$ of smooth functions on a compact $k$--dimensional manifold $S$ can be obtained in a natural way from pairs of differential forms on $M$ and $S$
by the hat pairing.
Special cases are the transgression map $\Om^p(M)\to\Om^{p-k}(\F(S,M))$ (hat pairing with a constant function) and the bar map $\Om^p(M)\to\Om^p(\F(S,M))$ (hat pairing with a volume form). We develop a hat calculus similar to the tilda calculus for non-linear Grassmannians \cite{HV04}. 
\end{abstract}

%%%%%%%%%%%%%%%%%%%%

\section{Introduction}

Pairs of differential forms on the finite dimensional manifolds $M$ and $S$ induce differential forms on the Fr\'echet manifold $\F(S,M)$ of smooth functions. More precisely, if $S$ is a compact oriented $k$--dimensional manifold, the hat pairing is:
\begin{gather*}
\Omega^p(M)\times\Omega^q(S)\to\Om^{p+q-k}(\F(S,M))\\
{\widehat{\om\cdot\al}=\fint_S\ev^*\om\wedge\pr^*\al},
\end{gather*}
where $\ev:S\x\F(S,M)\to M$ denotes the evaluation map, $\pr:S\x\F(S,M)\to S$ the projection and $\fint_S$ fiber integration. 
We show that the hat pairing is compatible with the canonical $\rm{Diff}(M)$ and ${\rm Diff}(S)$ actions on $\F(S,M)$, and with the exterior derivative.
As a consequence we obtain a hat pairing in cohomology.

The hat (transgression) map is the hat pairing with the constant function $1$, so it associates to any form $\om\in\Om^p(M)$ the form $\widehat{\om\cdot 1}=\widehat\om=\fint_S\ev^*\om\in\Om^{p-k}(\F(S,M))$. 
Since $\X(M)$ acts infinitesimally transitive on the open subset $\Emb(S,M)\subset\F(S,M)$ of embeddings of the $k$--dimensional oriented manifold $S$ into $M$ \cite{H76}, the expression of $\widehat\om$ at $f\in\Emb(S,M)$ is
\[
\widehat{\om}(X_1\o f,\dots, X_{p-k}\o f)=\int_Sf^*(i_{X_{p-k}}\dots i_{X_1}\om),\quad X_1,\dots,X_{p-k}\in\X(M).
\] 
When $S$ is the circle, then one obtains the usual transgression map 
with values in the space of $(p-1)$-forms on the free loop space of $M$. 

Let $\Gr_k(M)$ be the non-linear Grassmannian of $k$--dimensional oriented submanifolds of $M$. The tilda map associates to every
$\om\in\Om^p(M)$ a differential $(p-k)$-form on $\Gr_k(M)$
given by \cite{HV04}
\[
\tilde\om(\tilde Y_N^1,\dotsc,\tilde Y_N^{p-k})
=\int_Ni_{Y_N^{p-k}}\cdots i_{Y_N^1}\om,\quad\forall\tilde Y_N^1,\dots,\tilde Y_N^{p-k}\in \Ga(TN^\perp)=T_N\Gr_k(M),
\]
for $\tilde Y_N$
section of the orthogonal bundle $TN^\perp$ represented by the section $Y_N$ of $TM|_N$.
The natural map 
\[
\pi:\Emb(S,M)\to\Gr_k(M),\quad\pi(f)=f(S)
\]
provides a principal bundle with the group $\Diff_+(S)$ of orientation preserving diffeomorphisms of $S$ as structure group. 

The hat map on $\Emb(S,M)$ and the tilda map on $\Gr_k(M)$ 
are related by $\widehat\om=\pi^*\tilde\om$.
This is the reason why for the hat calculus one has similar properties to those for the tilda calculus.
The tilda calculus was used to study the non-linear Grassmannian of co-dimension two submanifolds as symplectic manifold \cite{HV04}.
We apply the hat calculus to
the hamiltonian formalism for $p$-branes and open $p$-branes
\cite{AS05} \cite{BZ05}.

The bar map $\bar\om=\widehat{\om\cdot\mu}$ is the hat pairing with a fixed volume form $\mu$ on $S$, so
\[
\bar\om(Y^1_f,\dots,Y^p_f)=\int_S \om(Y^1_f,\dots,Y^p_f)\mu,
\quad\forall Y^1_f,\dots,Y^{p}_f\in \Ga(f^*TM)=T_f \mathcal{F}(S,M).
\]
We use the bar calculus to study $\F(S,M)$ with symplectic form $\bar\om$ induced by a symplectic form $\om$ on $M$.
The natural actions of 
$\Diff_{ham}(M,\om)$ and $\Diff_{ex}(S,\mu)$, 
the group of hamiltonian diffeomorphisms of $M$ and the group of exact volume preserving diffeomorphisms of $S$, are two commuting hamiltonian actions on $\F(S,M)$. 
Their momentum maps 
form the dual pair for ideal incompressible fluid flow \cite{MW83} \cite{GBV09}.

We are grateful to Stefan Haller for extremely helpful suggestions.

%%%%%%%%%%%%%%%%%%%%%%%%%%

\section{Hat pairing}\label{calc}

We denote by $\mathcal{F}(S,M)$ the set of smooth functions from
a compact oriented $k$--dimensional manifold $S$ to a manifold $M$.
It is a Fr\'echet manifold in a natural way \cite{KM97}.
Tangent vectors at $f\in\mathcal{F}(S,M)$ are identified with vector fields on $M$ along $f$, \ie sections of the pull-back vector bundle $f^*TM$. 

Let $\ev:S\x\F(S,M)\to M$ be the evaluation map $\ev(x,f)=f(x)$
and $\pr:S\x\F(S,M)\to S$ the projection $\pr(x,f)=x$.
A pair of differential forms $\om\in\Om^p(M)$ and $\al\in\Om^q(S)$ determines a differential form $\widehat{\om\cdot\al}$ on $\F(S,M)$ by the fiber integral over $S$ (whose definition and properties are listed in the appendix) of the $(p+q)$-form $\ev^*\om\wedge\pr^*\al$ on $S\x\F(S,M)$:
\begin{equation}\label{use}
{\widehat{\om\cdot\al}=\fint_S\ev^*\om\wedge\pr^*\al}
\end{equation}
In this way we obtain a bilinear map called the {\it hat pairing}: 
\begin{equation*}\label{pair}
\Om^p(M)\x\Om^q(S)\to\Om^{p+q-k}(\F(S,M)).
\end{equation*}

An explicit expression of the hat pairing avoiding fiber integration is:
\begin{align}\label{ffff}
(\widehat{\om\cdot\al})_f(Y_f^1,\dots,Y_f^{p+q-k})
=\int_Sf^*(i_{Y_f^{p+q-k}}\dots i_{Y_f^1}(\om\o f))\wedge\al,
\end{align}
for $Y_f^1,\dots Y_f^{p+q-k}$ vector fields on $M$ along $f\in\F(S,M)$.
Here we denote by $f^*\be_f$ the "restricted pull-back" by $f$ of a section $\be_f$ of $f^*(\La^m T^*M)$, which is a differential $m$--form on $S$ given by
$f^*\be_f:x\in S\mapsto(\La^mT^*_xf)(\be_f(x))\in\La^mT_x^*S$,
where $T_x^*f:T^*_{f(x)}M\to T^*_xS$ denotes the dual of $T_xf$.

The fact that \eqref{use} and \eqref{ffff} provide the same differential form on $\F(S,M)$ can be deduced from the identity
\[
(\ev^*\om)_{(x,f)}(Y_f^1,\dots, Y_f^{p-k},X_x^1,\dots, X_x^k)
=f^*(i_{Y_f^{p-k}}\dots i_{Y_f^1}(\om\o f))(X_x^1,\dots, X_x^k)
\]
for $Y_f^1,\dots,Y_f^{p-k}\in T_f\F(S,M)$ and $X_x^1,\dots,X_x^k\in T_xS$. 

Since $\X(M)$ acts infinitesimally transitive on the open subset $\Emb(S,M)\subset\F(S,M)$ of embeddings of the $k$--dimensional oriented manifold $S$ into $M$, we express $\widehat\om$ at $f\in\Emb(S,M)$ as:
\begin{equation}\label{embf}
(\widehat{\om\cdot\al})_f(X_1\o f,\dots, X_{p+q-k}\o f)=\int_Sf^*(i_{X_{p+q-k}}\dots i_{X_1}\om)\wedge\al.
\end{equation}
One uses the fact that the "restricted pull-back" by $f$
of $i_{X_{p+q-k}\o f}\dots i_{X_1\o f}(\om\o f)$ is 
$f^*(i_{X_{p+q-k}}\dots i_{X_1}\om)$.

Next we show that the hat pairing is compatible with the exterior derivative of differential forms.

\begin{theo}\label{dddd}
The exterior derivative $\dd$ 
is a derivation for the hat pairing, \ie
\begin{equation}\label{deri}
\dd(\widehat{\om\cdot\al})= \widehat{(\dd\om)\cdot\al}+(-1)^{p}\widehat{\om\cdot\dd\al},
\end{equation}
where $\om\in\Om^p(M)$ and $\al\in\Om^q(S)$.
\end{theo}

\begin{proof}
Differentiation and fiber integration along the boundary free manifold $S$ commute, so
\begin{align*}
\dd(\widehat{\om\cdot\al})&= 
\dd\fint_S\ev^*\om\wedge\pr^*\al=\fint_S\dd(\ev^*\om\wedge\pr^*\al)\\
&=\fint_S\ev^*\dd\om\wedge\pr^*\al+(-1)^p\fint_S\ev^*\om\wedge\pr^*\dd\al
=\widehat{(\dd\om)\cdot\al}+(-1)^{p}\widehat{\om\cdot\dd\al}
\end{align*}
for all $\om\in\Om^p(M)$ and $\al\in\Om^q(S)$.
\end{proof}

The differential form $\widehat{\om\cdot\al}$ is exact if
$\om$ is closed and $\al$ exact
(or if $\al$ is closed and $\om$ exact).
In the special case $p+q=k$ these conditions imply that the function $\widehat{\om\cdot\al}$ on $\F(S,M)$ vanishes. 

\begin{coro}
The hat pairing induces a bilinear map on de Rham cohomology spaces
\begin{equation}\label{hh}
H^p(M)\x H^q(S)\to H^{p+q-k}(\F(S,M)).
\end{equation}
In particular there is a bilinear map
$$H^p(M)\x H^q(M)\to H^{p+q-k}(\Diff(M)).$$ 
\end{coro}

\begin{rema}
The cohomology group $H^q(S)$ is isomorphic to the  homology group $H_{k-q}(S)$ by Poincar\' e duality.
With the notation $n=k-q$, the hat pairing \eqref{hh} becomes
\[
H^p(M)\x H_n(S)\to H^{p-n}(\F(S,M)),
\]
and it is induced by the map $(\om,\si)\mapsto \fint_\si\ev^*\om$, for differential $p$-forms $\om$ on $M$
and $n$-chains $\si$ on $S$.
\end{rema}

%\medskip

If $S$ is a manifold with boundary, then formula \eqref{deri}
receives an extra term coming from integration over the boundary. Let $i_\pa:\pa S\to S$ be the inclusion and $r_\pa:\F(S,M)\to\F(\pa S,M)$ the restriction map.

\begin{prop}\label{inde}
The identity
\begin{equation}\label{unu}
\dd(\widehat{\om\cdot\al})= \widehat{(\dd\om)\cdot\al}+(-1)^{p}\widehat{\om\cdot\dd\al}
+(-1)^{p+q-k}r_\pa^*(\widehat{\om\cdot i_\pa^*\al}^\pa)
\end{equation}
holds for $\om\in\Om^p(M)$ and $\al\in\Om^q(S)$, where
the upper index $\pa$ assigned to the hat means the pairing 
$$
\Om^p(M)\x\Om^q(\pa S)\to\Om^{p+q-k+1}(\F(\pa S,M)).
$$
\end{prop}

\begin{proof}
For any differential $n$--form $\be$ on $S\x\F(S,M)$,
%with $S$ a $k$-dimensional compact manifold,
the identity
\[
\dd\fint_S\be-\fint_S\dd\be=(-1)^{n-k}\fint_{\pa S}(i_\pa\x 1_{\F(S,M)})^*\be
\]
holds because of the identity \eqref{r4} from the appendix. The obvious formulas
\[
\pr\o(i_\pa\x 1_{\F(S,M)})=i_\pa\o\pr_\pa,\quad \ev\o(i_\pa\x 1_{\F(S,M)})=\ev_\pa,
\]
for $\ev_\pa:\pa S\x\F(S,M)\to M$ and
$\pr_\pa:\pa S\x\F(S,M)\to\pa S$, are used to compute
\begin{align*}
\dd&(\widehat{\om\cdot\al})= 
\dd\fint_S\ev^*\om\wedge\pr^*\al\\
&=\fint_S\dd(\ev^*\om\wedge\pr^*\al)
+(-1)^{p+q-k}\fint_{\pa S}(i_\pa\x 1_{\F(S,M)})^*(\ev^*\om\wedge\pr^*\al)\\
&=\fint_S\ev^*\dd\om\wedge\pr^*\al+(-1)^p\fint_S\ev^*\om\wedge\pr^*\dd\al
+(-1)^{p+q-k}\fint_{\pa S}\ev_\pa^*\om\wedge\pr_\pa^*i_\pa^*\al\\
&=\widehat{(\dd\om)\cdot\al}+(-1)^{p}\widehat{\om\cdot\dd\al}
+(-1)^{p+q-k}{r_\pa}^*(\widehat{\om\cdot i_\pa^*\al}^\pa),
\end{align*}
thus obtaining the requested identity.
\end{proof}

%%%%%%%%%%%%%%%%%%%%%%%%%%%%%%%

%Next we show some functorial properties of the hat pairing.

\paragraph{Left $\Diff(M)$ action.}
The natural left action of the group of diffeomorphisms $\Diff(M)$ on $\F(S,M)$ 
is $\ph\cdot f=\ph\o f$. The infinitesimal action of $X\in\X(M)$ is  the vector field $\bar X$ on $\F(S,M)$: 
\[
\bar X(f)=X\o f,\quad\forall f\in\F(S,M).
\] 
We denote by $\bar\ph$ the diffeomorphism of $\F(S,M)$ induced by 
the action of $\ph\in\Diff(M)$, so $\bar\ph(f)=\ph\o f$ is the push-forward by $\ph$.

\begin{prop}\label{pppp}
Given $\om\in\Om^p(M)$ and $\al\in\Om^q(S)$, the identity
\begin{align}\label{star}
\bar\ph^*\widehat{\om\cdot\al}&=\widehat{(\ph^*\om)\cdot\al}
\end{align}
and its infinitesimal version 
\begin{align}\label{el}
L_{\bar X}\widehat{\om\cdot\al}&=\widehat{(L_X\om)\cdot\al}
\end{align}
hold for all $\ph\in\Diff(M)$ and $X\in\X(M)$.
\end{prop}

\begin{proof}
Using the expression \eqref{use} of the hat pairing 
and identity \eqref{r1} from the appendix, we have:
\begin{align*}
\bar\ph^*\widehat{\om\cdot\al}&
=\bar\ph^*\fint_S\ev^*\om\wedge\pr^*\al
=\fint_S(1_S\x\bar\ph)^*(\ev^*\om\wedge\pr^*\al)\\
&=\fint_S\ev^*\ph^*\om\wedge\pr^*\al
=\widehat{(\ph^*\om)\cdot\al},
\end{align*}
since $\pr\o(1_S\x\bar\ph)=\pr$ and $\ev\o(1_S\x\bar\ph)=\ph\o\ev$.
\end{proof}

A similar result is obtained for any smooth map $\et\in\F(M_1,M_2)$
and its push-forward $\bar\et:\F(S,M_1)\to\F(S,M_2)$, $\bar\et(f)=\et\o f$:
\begin{equation*}
\bar\et^*\widehat{\om\cdot\al}=\widehat{\et^*\om\cdot\al}, 
\end{equation*}
for all $\om\in\Om^p(M_2)$ and $\al\in\Om^q(S)$.

\begin{lemm}\label{ins}
For all vector fields $X\in\X(M)$, the identity $i_{\bar X}\widehat{\om\cdot\al}=\widehat{(i_X\om)\cdot\al}$ holds.
\end{lemm}

\begin{proof}
The vector field $0_S\x \bar X$ on $S\x \F(S,M)$ is $\ev$-related to the vector field $X$ on $M$, so
\begin{align*}
i_{\bar X}\widehat{\om\cdot\al}&
=i_{\bar X}\fint_S \ev^*\om\wedge\pr^*\al
=\fint_Si_{0_S\x \bar X}(\ev^*\om\wedge\pr^*\al)\\
&=\fint_S\ev^*(i_X\om)\wedge\pr^*\al
%+(-1)^p\int_S\ev^*\om\wedge\pr^* (i_{0_S}\al)
=\widehat{(i_X\om)\cdot\al}.
\end{align*}
At step two we use formula \eqref{r3} from the appendix.
\end{proof}

%%%%%%%%%%%%%%%%%%%%%%%%%%%%%

\paragraph{Right $\Diff(S)$ action.}
The natural right action of the diffeomorphism group $\Diff(S)$ on $\F(S,M)$ 
can be transformed into a left action 
by
$\ps\cdot f=f\o\ps^{-1}$. The infinitesimal action of $Z\in\X(S)$ is the vector field $\hat Z$ on $\F(S,M)$:
\[
\widehat Z(f)=-Tf\o Z,\quad\forall f\in\F(S,M).
\]
We denote by $\widehat\ps$ the diffeomorphism of $\F(S,M)$ induced by 
the action of $\ps$, so $\widehat\ps(f)=f\o\ps^{-1}$
is the pull-back by $\ps^{-1}$.

\begin{prop}\label{five}
Given $\om\in\Om^p(M)$ and $\al\in\Om^q(S)$, the identity
\begin{align*}
\widehat\ps^*\widehat{\om\cdot\al}&=\widehat{\om\cdot\ps^*\al}
\end{align*}
and its infinitesimal version 
\begin{align*}
L_{\widehat Z}\widehat{\om\cdot\al}&=\widehat{\om\cdot L_Z\al}
\end{align*}
hold for all orientation preserving $\ps\in\Diff(S)$ and $Z\in\X(S)$.
\end{prop}

\begin{proof}
The obvious identities $\ev\o(1_S\x \widehat\ps)=\ev\o(\ps^{-1}\x 1_\F)$,
$\pr\o(1_S\x \widehat\ps)=\pr$ and $\pr\o(\ps\x 1_\F)=\ps\o\pr$
are used in the computation
\begin{align*}
\widehat\ps^*\widehat{\om\cdot\al}&
=\widehat\ps^*\fint_S\ev^*\om\wedge\pr^*\al
=\fint_S(1_S\x\widehat\ps)^*(\ev^*\om\wedge\pr^*\al)\\
&=\fint_S\big((\ps^{-1}\x 1_\F)^*\ev^*\om\big)\wedge\pr^*\al
=\fint_S\ev^*\om\wedge(\ps\x 1_\F)^*\pr^*\al\\
&=\fint_S\ev^*\om\wedge\pr^*\ps^*\al
=\widehat{\om\cdot\ps^*\al},
\end{align*}
together with formula \eqref{r2} from the appendix at step four.
\end{proof}

\begin{lemm}\label{ins2}
The identity
$i_{\widehat Z}\widehat{\om\cdot\al}=(-1)^{p}\widehat{\om\cdot i_Z\al}$ holds for all vector fields $Z\in\X(S)$, 
if $\om\in\Om^p(M)$.
\end{lemm}

\begin{proof}
The infinitesimal version of the first identity in the proof of proposition \ref{five} is $T\ev.(0_S\x \widehat Z)=T\ev.(-Z\x 0_{\F(S,M)})$, so we compute:
\begin{align*}
i_{\widehat Z}\widehat{\om\cdot\al}&
=i_{\widehat Z}\fint_S \ev^*\om\wedge\pr^*\al
=\fint_S i_{0_S\x \widehat Z}(\ev^*\om\wedge\pr^*\al)\\
&=\fint_S (i_{0_S\x \widehat Z}\ev^*\om)\wedge\pr^*\al
=\fint_S(i_{-Z\x 0_{\F(S,M)}}\ev^*\om)\wedge\pr^*\al\\
&=\fint_Si_{-Z\x 0_{\F(S,M)}}(\ev^*\om\wedge\pr^*\al)
-\fint_S(-1)^p\ev^*\om\wedge i_{-Z\x 0_{\F(S,M)}}\pr^*\al\\
&=(-1)^{p}\fint_S\ev^*\om\wedge\pr^*(i_Z\al)
=(-1)^{p}\widehat{\om\cdot i_Z\al}.
\end{align*}
At step two we use formula \eqref{r3} from the appendix.
\end{proof}

%%%%%%%%%%%%%%%%%%%%%
%%%%%%%%%%%%%%%%%%%%%%%%%%%%%%%%%

\section{Tilda map and hat map}\label{s3}

Let $\Gr_k(M)$ be the non-linear Grassmannian (or differentiable Chow variety) 
of compact oriented $k$--dimensional submanifolds of $M$. 
It is a Fr\' echet manifold \cite{KM97} and the tangent space 
at $N\in\Gr_k(M)$ can be identified with the space of smooth sections of the normal bundle $TN^\perp=(TM|_N)/TN$.
The tangent vector at $N$ determined by the section $Y_N\in\Ga(TM|_N)$ is
denoted by $\tilde Y_N\in T_N\Gr_k(M)$.

The {\it tilda map} \cite{HV04} associates to any $p$--form $\om$ on $M$ a $(p-k)$--form
$\tilde\om$ on $\Gr_k(M)$ by:
\begin{equation}\label{tide}
\tilde\om_N(\tilde Y_N^1,\dotsc,\tilde Y_N^{p-k})
=\int_Ni_{Y_N^{p-k}}\cdots i_{Y_N^1}\om.
\end{equation}
Here all $\tilde Y_N^j$ are tangent vectors at $N\in\Gr_k(M)$, \ie sections
of $TN^\perp$ represented by sections $Y_N^j$ of $TM|_N$. 
Then $i_{Y_N^{p-k}}\cdots i_{Y_N^1}\om\in\Omega^k(N)$
does not depend on representatives $Y_N^j$ of $\tilde Y_N^j$, and integration is well defined since
$N\in\Gr_k(M)$ comes with an orientation. 

Let $S$ be a compact oriented $k$--dimensional manifold. The {\it hat map} is the hat pairing with the constant function $1\in\Om^0(S)$. It associates to any form $\om\in\Om^p(M)$ the form $\widehat\om\in\Om^{p-k}(\F(S,M))$:
\begin{equation}\label{hatm}
\widehat\om=\widehat{\om\cdot 1}=\fint_S\ev^*\om.
\end{equation}
On the open subset $\Emb(S,M)\subset\F(S,M)$ of embeddings, formula \eqref{ffff} gives 
\begin{equation}\label{exte}
\widehat{\om}(X_1\o f,\dots, X_{p-k}\o f)=\int_Sf^*(i_{X_{p-k}}\dots i_{X_1}\om).
\end{equation}

\begin{rema}
The hat map induces a transgression on cohomology spaces 
\[
H^p(M)\to H^{p-k}(\F(S,M)).
\]
When $S$ is the circle, then one obtains the usual transgression map 
with values in the $(p-1)$-th cohomology space of the free loop space of $M$.
\end{rema}

Let $\pi$ denote the natural map 
\[
\pi:\Emb(S,M)\to\Gr_k(M),\quad\pi(f)=f(S).
\]
where the orientation on $f(S)$ is chosen such that the diffeomorphism $f:S\to f(S)$ is orientation preserving. 
The image $\pi(\Emb(S,M))$ is the manifold $\Gr_k^S(M)$ 
of $k$--dimensional submanifolds of $M$ of type $S$.
Then $\pi:\Emb(S,M)\to\Gr_k^S(M)$
is a principal bundle over $\Gr_k^S(M)$  with structure group $\Diff_+(S)$, the group of orientation preserving diffeomorphisms of $S$. 

Note that there is a natural action of the group $\Diff(M)$ on 
the non-linear Grassmannian $\Gr_k(M)$ given by $\ph\cdot N=\ph(N)$. Let $\tilde\ph$ be the diffeomorphism of $\Gr_k(M)$ induced by the action of $\ph\in\Diff(M)$. Then $\tilde\ph\o\pi=\pi\o\bar\ph$ for the restriction of $\bar\ph(f)=\ph\o f$ to a diffeomorphism of 
$\Emb(S,M)\subset\F(S,M)$. As a consequence, the infinitesimal 
generators for the $\Diff(M)$ actions on $\Gr_k(M)$ and on $\Emb(S,M)$ are $\pi$--related. This means that for all $X\in\X(M)$, the vector fields $\tilde X$ on $\Gr_k(M)$ given by $\tilde X(N)=X|_N$ and $\bar X$ on $\Emb(S,M)$ given by $\bar X(f)=X\o f$ are $\pi$--related.

\begin{prop}
The hat map on $\Emb(S,M)$ and the tilda map on $\Gr_k(M)$ are related by 
$\widehat\om=\pi^*\tilde\om$,
for any $k$--dimensional oriented manifold $S$.
\end{prop}

\begin{proof} 
For the proof we use the fact that $\X(M)$ acts infinitesimally transitive on $\Emb(S,M)$, so  $T_f\Emb(S,M)=\{X\o f:X\in\X(M)\}$.
With \eqref{tide} and \eqref{exte} we compute:
\begin{multline*}
(\pi^*\tilde\om)_f(X_1\o f,\dots, X_{p-k}\o f)
=\tilde\om_{f(S)}(X_1|_{f(S)},\dots,X_{p-k}|_{f(S)})\\
=\int_{f(S)} i_{X_{p-k}}\dots i_{X_1}\om=\int_Sf^*(i_{X_{p-k}}\dots i_{X_1}\om)=\widehat\om_f(X_1\o f,\dots,X_{p-k}\o f),
\end{multline*}
since $\bar X$ and $\tilde X$ are $\pi$--related.
\end{proof}

From the properties of the hat pairing presented in proposition \ref{pppp}, lemma \ref{ins} and theorem \ref{dddd}, a hat calculus
follows easily:

\begin{prop}\label{cor3}
For any $\om\in\Om^p(M)$, $\ph\in\Diff(M)$, $X\in\X(M)$, and $\et\in\F(M', M)$ with push-forward $\bar\et:\F(S,M')\to\F(S,M)$, the following identities hold: 
\begin{enumerate}
\item $\bar\ph^*\widehat{\om}=\widehat{\ph^*\om}$ and $\bar\et^*\widehat\om=\widehat{\et^*\om}$
\item $L_{\bar X}\widehat\om=\widehat{L_X\om}$
\item $i_{\bar X}\widehat\om=\widehat{i_X\om}$
\item $\dd\widehat\om=\widehat{\dd\om}$. 
\end{enumerate}
\end{prop}

\begin{rema}
If $S$ is a manifold with boundary, then the formula 4. above
receives an extra term coming from integration over the boundary $\pa S$ as in proposition \ref{inde}:
\begin{equation}\label{bond}
\dd\widehat\om=\widehat{\dd\om}+(-1)^{p-k}r_\pa^*\widehat\om^\pa
\end{equation}
for $\om\in\Om^p(M)$. As before, $r_\pa:\F(S,M)\to\F(\pa S,M)$ denotes the restriction map on functions
and $\om\in\Om^p(M)\mapsto\widehat\om^\pa\in\Om^{p-k+1}(\F(\pa S,M))$. 
\end{rema}

Now the properties of the tilda calculus follow imediately from proposition \ref{cor3}. 

\begin{prop}\cite{HV04}\label{cor4}
For any $\om\in\Om^p(M)$, $\ph\in\Diff(M)$ and $X\in\X(M)$, the following identities hold: 
\begin{enumerate}
\item $\tilde\ph^*\tilde{\om}=\widetilde{\ph^*\om}$
\item $L_{\tilde X}\tilde\om=\widetilde{L_X\om}$
\item $i_{\tilde X}\tilde\om=\widetilde{i_X\om}$
\item $\dd\tilde\om=\widetilde{\dd\om}$.
\end{enumerate}
\end{prop}

\begin{proof}
We verify the identities 1. and 4. From relation 1. from proposition \ref{cor3} we get that
\begin{align*}
\pi^*\tilde\ph^*\tilde\om&
=\bar\ph^*\pi^*\tilde\om
=\bar\ph^*\widehat\om
=\widehat{\ph^*\om}
=\pi^*\widetilde{\ph^*\om},
\end{align*}
and this implies the first identity.
Using identity 4. from proposition \ref{cor3} we compute
$$
\pi^*\dd\tilde\om=\dd\pi^*\tilde\om
=\dd \widehat\om
=\widehat{\dd\om}=\pi^*\widetilde{\dd\om},
$$
which shows the last identity.
\end{proof}
%%%%%%%%%%%%%

\subsection*{Hamiltonian formalism for $p$-branes}

In this section we show how the hat calculus appears in
the hamiltonian formalism for $p$-branes and open $p$-branes
\cite{AS05} \cite{BZ05}.

Let $S$ be a compact oriented $p$-dimensional manifold.
The phase space for the $p$-brane world volume
$S\x\RR$ is the cotangent bundle $T^*\F(S,M)$,
where the canonical symplectic form is twisted.
The twisting consists in adding a magnetic term, namely the pull-back of a closed 2-form on the base manifold, to the canonical symplectic form on a cotangent bundle \cite{MR99}. These twisted symplectic forms appear also in cotangent bundle reduction. 

We consider a closed differential form $H\in\Om^{p+2}(M)$.
Since $\dim S=p$, the hat map \eqref{hatm}
provides a closed 2-form $\widehat H$ on $\F(S,M)$.
If $\pi_\F:T^*\F(S,M)\to\F(S,M)$ denotes the canonical projection,
the twisted symplectic form on $T^*\F(S,M)$ is
\[
\Om_H=-\dd\Th_\F+\frac{1}{2}\pi_\F^*\widehat H,
\]
where $\Th_\F$ is the canonical 1-form on $T^*\F(S,M)$.

For the description of open branes one considers 
a compact oriented $p$-dimensional manifold $S$ with boundary $\pa S$
and a submanifold $D$ of $M$.
The phase space is in this case the cotangent bundle $T^*\F_D(S,M)$ over the manifold \cite{M80}
\[
\F_D(S,M)=\{f:S\to M|f(\pa S)\subset D\}.
\]
The twisting of the canonical symplectic form is done with a closed 
differential form $H\in\Om^{p+2}(M)$ with $i^*H=\dd B$ for some 
$B\in\Om^{p+1}(D)$, where $i:D\to M$ denotes the inclusion.
The twisted symplectic form on $T^*\F_D(S,M)$ is
\[
\Om_{(H,B)}=-\dd\Th_{\F_D}+\frac12\pi_{\F_D}^*(\widehat H-\pa^*\widehat B^\pa)
\]
with $\pa:\F_D(S,M)\to\F(\pa S,D)$ 
the restriction map and $\pi_{\F_D}:T^*\F_D(S,M)\to\F_D(S,M)$.
To distinguish between the hat calculus 
for $\F(S,M)$
and the hat calculus for $\F(\pa S,M)$, we 
denote $\widehat \ ^\pa:\Om^n(M)\to\Om^{n-p+1}(\F(\pa S,M))$.

The only thing we have to verify is the closedness of $\widehat H-\pa^*\widehat B^\pa$.
We first notice that \eqref{bond} implies 
$\dd\widehat H=\widehat{\dd H}+r_\pa^*\widehat H^\pa$, 
where 
$r_\pa:\F(S,M)\to\F(\pa S,M)$ denotes the restriction map,
and identity 4 from proposition \ref{cor3}
implies $\widehat{\dd B}^\pa=\dd\widehat B^\pa$.
On the other hand identity 1 from proposition \ref{cor3} ensures that 
$\widehat{i^*H} ^\pa=\bar i^*\widehat H^\pa$,
with $\bar i:\F(\pa S,D)\to\F(\pa S,M)$ denoting the push-forward by $i:D\to M$.
Knowing that $r_\pa=\bar i\o\pa$, we compute:
\begin{align*}
\dd\widehat H 
=\widehat{\dd H}+r_\pa^*\widehat H^\pa
=\pa^*\bar i^*\widehat H^\pa=\pa^*\widehat{i^*H}^\pa
=\pa^*\widehat{\dd B}^\pa=\dd\pa^*\widehat B^\pa,
\end{align*}
so the closed 2--form $\widehat H-\pa^*\widehat B^\pa$ provides
a twist for the canonical symplectic form on the cotangent bundle 
$T^*\F_D(S,M)$.

%%%%

\subsection*{Non-linear Grassmannians as symplectic manifolds}

In this subsection we recall properties of the co-dimension two non-linear Grassmannian as a symplectic manifold.

\begin{prop}\cite{I96}
Let $M$ be a closed $m$--dimensional manifold with volume form $\nu$. 
The tilda map provides a symplectic form $\tilde\nu$ on $\Gr_{m-2}(M)$
$$
\tilde\nu_N(\tilde X_N, \tilde Y _N)=\int_Ni_{ Y _N}i_{X_N}\nu,
$$
for $\tilde X_N$ and $\tilde Y _N$ sections of $TN^\perp$ determined by sections $X_N$ and $Y_N$ of $TM|_N$. 
\end{prop}

\begin{proof}
The 2--form $\tilde\nu$ is closed since $\dd\tilde\nu=\widetilde{\dd\nu}$ by the tilda calculus. To verify that it is also (weakly) non-degenerate,
let $X_N$ be an arbitrary vector field along $N$ such that $\int_Ni_{ Y _N}i_{X_N}\nu=0$ for all vector fields $Y_N$ along $N$. Then $X_N$ must be tangent to $N$, so $\tilde X_N=0$.
\end{proof}

In dimension $m=3$ the symplectic form $\tilde\nu$ is known as the 
Marsden--Weinstein 
symplectic from on the space of unparameterized oriented links, see
\cite{MW83} \cite{B93}.

\paragraph{Hamiltonian $\Diff_{ex}(M,\nu)$ action.}
The action of the group $\Diff(M,\nu)$ of volume preserving
diffeomorphisms of $M$ on $\Gr_{m-2}(M)$ 
preserves the symplectic form $\tilde\nu$: 
\[
\tilde\ph^*\tilde\nu=\widetilde{\ph^*\nu}=\tilde\nu,
\quad\forall\ph\in\Diff(M,\nu).
\]

The subgroup $\Diff_{ex}(M,\nu)$ of exact volume preserving diffeomorphisms acts in a hamiltonian way on the symplectic manifold $(\Gr_{m-2}(M),\tilde\nu)$.
Its Lie algebra is $\X_{ex}(M,\nu)$, the Lie algebra of exact divergence free vector fields, \ie vector fields $X_\al$
such that $i_{X_\al}\nu=\dd\al$ for a potential form $\al\in\Om^{m-2}(M)$.
The infinitesimal action of $X_\al$
is the vector field $\tilde X_\al$.
By the tilda calculus $\tilde\al\in\F(\Gr_{m-2}(M))$ is a hamiltonian function for the hamiltonian vector field $\tilde X_\al$:
\[
i_{\tilde X_\al}\tilde\nu=\widetilde{i_{X_\al}\nu}=\widetilde{\dd\al}=\dd\tilde\al.
\]
It depends on the particular choice of the potential $\al$ of $X_\al$.
A fixed continuous right inverse $b:\dd\Om^{m-2}(M)\to\Om^{m-2}(M)$
to the differential $\dd$ picks up a potential $b(\dd\al)$ of $X_\al$.
The corresponding momentum map is:
\begin{align*}
\mathbf{J}:\M\to \X_{ex}(M,\nu)^*,
\quad\langle\mathbf{J}(N),X_\al\rangle
=\widetilde{b(\dd\al)}(N)=\int_N b(\dd\al).
\end{align*}

On the connected component $\M$ of $N\in\Gr_{m-2}(M)$, 
the non-equivariance of $\mathbf{J}$ is measured 
by the Lie algebra 2--cocycle on $\X_{ex}(M,\nu)$
\begin{align*}
\si_N(X,Y)=&\langle\mathbf{J}(N),[X,Y]^{op}\rangle
-\tilde\nu(\tilde X,\tilde Y)(N)
=\widetilde{(b\dd i_Yi_X\nu)}(N)-\widetilde{(i_Yi_X\nu)}(N)\\
&=\widetilde{(Pi_Xi_Y\nu)}(N)
=\int_NPi_Xi_Y\nu.
\end{align*}
Here
$P=1_{\Om^{m-2}(M)}-b\o \dd$
is a continuous linear projection on the subspace of closed $(m-2)$-forms and $(X,Y)\mapsto[Pi_Yi_X\nu]\in H^{m-2}(M)$ is the universal Lie algebra 2--cocycle on $\X_{ex}(M,\nu)$ \cite{R95}. 
The cocycle $\si_{N}$ is cohomologous to the Lichnerowicz cocycle
\begin{equation}\label{lich}
\si_\et(X,Y)=\int_M\et(X,Y)\nu,
\end{equation}
where $\et$ is a closed 2-form Poincar\'e dual to $N$ \cite{V09}.

If $\nu$ is an integral volume form, then 
$\si_N$ is integrable \cite{I96}. 
The connected component $\M$ of $\Gr_{m-2}(M)$ is a coadjoint orbit 
of a 1--dimen\-sional central Lie group extension of $\Diff_{ex}(M,\nu)$ integrating $\si_{N}$, 
and $\tilde\nu$ is the Kostant-Kirillov-Souriau symplectic form. 
\cite{HV04}.

%%%%%%%%%%%%%%%%%
%%%%%%%%%%%%%%%%%%%%%%%%%

\section{Bar map}\label{s4}

When a volume form $\mu$ on the compact $k$--dimensional manifold $S$ is given, one can associate to each differential $p$-form on $M$ a differential $p$-form on $\F(S,M)$ 
\[
\bar\om(Y^1_f,\dots,Y^p_f)=\int_S \om(Y^1_f,\dots,Y^p_f)\mu,
\quad\forall Y^i_f\in T_f \mathcal{F}(S,M),
\]
where $\om(Y^1_f,\dots,Y^p_f):x\mapsto \om_{f(x)}(Y_f^1(x),\dots, Y_f^p(x))$ defines a smooth function on $S$.
In this way a {\it bar map} is defined.
Formula \eqref{ffff} assures that this bar map is just the hat pairing of differential forms on $M$ with the volume form $\mu$
\begin{equation}\label{barb}
\bar\om=\widehat{\om\cdot\mu}=\fint_S\ev^*\om\wedge\pr^*\mu.
\end{equation}

From the properties of the hat pairing presented in proposition \ref{pppp}, lemma \ref{ins} and theorem \ref{dddd}, one can develop a bar calculus.
\begin{prop}\label{cor2}
For any $\om\in\Om^p(M)$, $\ph\in\Diff(M)$ and $X\in\X(M)$, the following identities hold: 
\begin{enumerate}
\item $\bar\ph^*\bar{\om}=\overline{\ph^*\om}$
\item $L_{\bar X}\bar\om=\overline{L_X\om}$
\item $i_{\bar X}\bar\om=\overline{i_X\om}$
\item $\dd\bar\om=\overline{\dd\om}$.
\end{enumerate}
\end{prop}

%%%%

\subsection*{$\F(S,M)$ as symplectic manifold}

Let $(M,\om)$ be a connected symplectic manifold and $S$ a compact $k$--dimensional manifold with a fixed volume form $\mu$, normalized such that $\int_S\mu=1$. 
The following fact is well known:

\begin{prop}
The bar map provides a symplectic form $\bar\om$ on $\F(S,M)$:
$$
\bar\om_f(X_f,Y_f)=\int_S\om(X_f,Y_f)\mu.
$$ 
\end{prop}

\begin{proof}
That $\bar\om$ is closed follows from the bar calculus: $\dd\bar\om=\overline{\dd\om}=0$.
The (weakly) non-degeneracy of $\bar\om$ can be verified as follows.
If the vector field $X_f$ on $M$ along $S$ is non-zero, then $X_f(x)\ne 0$ for some $x\in S$. Because $\om$ is non-degenerate, one can find another vector field $Y_f$ along $f$ such that $\om(X_f,Y_f)$ is a bump function on $S$. Then $\bar\om(X_f,Y_f)=\int_S\om(X_f,Y_f)\mu\ne 0$, so $X_f$ does not belong to the kernel of $\bar\om$,
thus showing that the kernel of $\bar\om$ is trivial.
\end{proof}

%%%%%

\paragraph{Hamiltonian action on $M$.}

Let $G$ be a Lie group acting in a hamiltonian way on $M$
with momentum map $J:M\to\g^*$.
Then $\F(S,M)$ inherits a $G$-action:
$(g\cdot f)(x)=g\cdot(f(x))$ for any $x\in S$.
The infinitesimal generator is $\xi_\F=\bar\xi_M$
for any $\xi\in\g$,
where $\xi_M$ denotes the infinitesimal generator for the $G$-action on $M$.
The bar calculus shows quickly that $G$ acts in a hamiltonian
way on $\F(S,M)$ with momentum map 
\[
\mathbf{J}=\bar J:\F(S,M)\to \g^*,\quad 
\bar J(f)=\int_S(J\o f)\mu,\quad\forall f\in\F(S,M).
\]
Indeed, for all $\xi\in\g$
\[
i_{\xi_\F}\bar\om=i_{\bar\xi_M}\bar\om=\overline{i_{\xi_M}\om}
=\overline{\dd\langle J,\xi\rangle}=\dd\langle\bar J,\xi\rangle.
\]

Let $M$ be connected and let $\si$ be the $\RR$-valued Lie algebra 2--cocycle on $\g$
measuring the non-equivariance of $J$, \ie
\[
\si(\xi,\et)={\langle J(x),[\xi,\et]\rangle}
-{\om(\xi_M,\et_M)}(x),\quad x\in M,
\]
(both terms are hamiltonian function for the vector field $[\xi,\et]_M=-[\xi_M,\et_M]$). Then the non-equivariance of $\mathbf{J}=\bar J$ is also measured by $\si$: for all $f\in\F(S,M)$
\[
\langle\bar J(f),[\xi,\et]\rangle-\bar\om(\xi_\F,\et_\F)(f)
=\overline{\langle J,[\xi,\et]\rangle}(f)
-\overline{\om(\xi_M,\et_M)}(f)
%=\overline{\si(\xi,\et)}(f)
=\si(\xi,\et).
\]

%%%%%

\paragraph{Hamiltonian $\Diff_{ham}(M,\om)$ action.} 
The action of the group $\Diff(M,\om)$ of symplectic diffeomorphisms preserves the symplectic form $\bar\om$:
\begin{equation*}
\bar\ph^*\bar\om=\overline{\ph^*\om}=\bar\om,\quad \forall\ph\in\Diff(M,\om).
\end{equation*}

The subgroup $\Diff_{ham}(M,\om)$ of hamiltonian diffeomorphisms of $M$ acts in a hamiltonian way on the symplectic manifold $\F(S,M)$.
The infinitesimal action of $X_h\in\X_{ham}(M,\om)$, $h\in\F(M)$,
is the hamiltonian vector field $\bar X_h$ on $\F(S,M)$ with hamiltonian function $\bar h$.
This follows by the bar calculus:
\begin{gather*}
\dd\bar h=\overline{\dd h}=\overline{i_{X_h}\om}=i_{\bar X_h}\bar\om.
\end{gather*}

The hamiltonian function $\bar h$ of $\bar X_h$ depends on the particular choice of the hamiltonian function $h$. To solve this problem we fix a point $x_0\in M$ and we choose the unique hamiltonian function $h$ with $h(x_0)=0$,
since $M$ is connected.
The corresponding momentum map is 
\begin{gather*}
\mathbf{J}:\mathcal{F}(S,M)\to \X_{ham}(M,\om)^*,\quad\langle\mathbf{J}(f),X_h\rangle=\bar h(f)=\int_S(h\o f)\mu.
\end{gather*}

The Lie algebra 2--cocycle on $\X_{ham}(M,\om)$ measuring
the non-equivariance of the momentum map is
\[
\si(X,Y)=-\om(X,Y)(x_0),
\]
by the bar calculus
\begin{align*}
\si(X,Y)(f)&=\langle\mathbf{J}(f),[X,Y]^{op}\rangle
-\bar\om(X_\F,Y_\F)(f)\\
&=\overline{\om(X,Y)-\om(X,Y)(x_0)}(f)-\bar\om(\bar X,\bar Y)(f)
=-\om(X,Y)(x_0).
\end{align*}
This is a Lie algebra cocycle describing the central 
extension 
$$0\to\RR\to\F(M)\to\X_{ham}(M,\om)\to 0$$ 
where $\F(M)$ is enowed with the canonical Poisson bracket. 
A group cocycle on $\Diff_{ham}(M,\om)$  integrating the Lie algebra cocycle $\si$ if $\om$ exact 
is studied in \cite{ILM06}.

%%%%

\paragraph{Hamiltonian $\Diff_{ex}(S,\mu)$ action.}
The (left) action of the group $\Diff(S,\mu)$ of volume preserving diffeomorphisms preserves the symplectic form $\bar\om$:
\begin{equation*}
\widehat\ps^*\bar\om=\widehat\ps^*\widehat{\om\cdot\mu}=\widehat{\om\cdot\ps^*\mu}=\widehat{\om\cdot\mu}=\bar\om,\quad\forall\ps\in\Diff(S,\mu).
\end{equation*}

The subgroup $\Diff_{ex}(S,\mu)$ of exact volume preserving diffeomorphisms acts in a hamiltonian way on the symplectic manifold $\F(S,M)$.
The infinitesimal action of the exact divergence free vector field $X_\al\in\X_{ex}(S,\mu)$ with potential form $\al\in\Om^{k-2}(S)$
is the hamiltonian vector field $\widehat X_\al$ on $\F(S,M)$ with hamiltonian function $\widehat{\om\cdot\al} $.
Indeed, from $i_{X_\al}\mu=\dd\al$ follows by the hat calculus that
\begin{gather*}
\dd(\widehat{\om\cdot\al})=\widehat{\dd\om\cdot\al}+\widehat{\om\cdot\dd\al}=\widehat{\om\cdot i_{X_\al}\mu}=i_{\widehat X_\al}\widehat{\om\cdot\mu}=i_{\widehat X_\al}\bar\om.
\end{gather*}

If the symplectic form $\om$ is exact, then the corresponding momentum map is
\begin{gather*}
\mathbf{J}:\mathcal{F}(S,M)\to \X_{ex}(S,\mu)^*,\quad
\langle\mathbf{J}(f),X_\al\rangle
=\widehat{(\om\cdot \al)}(f)
=\int_Sf^*\om\wedge \al.
\end{gather*}
It takes values in the regular part of $\X_{ex}(S,\mu)^*$,
which can be identified with $\dd\Om^1(S)$, so we can write $\mathbf{J}(f)=f^*\om$
under this identification.

In general the hamiltonian function $\widehat{\om\cdot\al}$ of $\widehat X_\al$ depends on the particular choice of the potential form $\al$ of $X_\al$.
To fix this problem we consider as in Section \ref{s3} a
continuous right inverse $b:\dd\Om^{m-2}(M)\to\Om^{m-2}(M)$
to the differential $\dd$, so $b(\dd\al)$ is a potential for $X_\al$.
The corresponding momentum map is
\begin{gather*}
\mathbf{J}:\mathcal{F}(S,M)\to \X_{ex}(S,\mu)^*,\quad
\langle\mathbf{J}(f),X_\al\rangle
=\widehat{(\om\cdot b\dd\al)}(f)
=\int_Sf^*\om\wedge b(\dd\al).
\end{gather*}

On a connected component $\F$ of $\F(S,M)$, 
the non-equivariance of $\mathbf{J}$ is measured 
by the Lie algebra 2--cocycle
\begin{align*}
\si_\F(X,Y)&=\langle\mathbf{J}(f),[X,Y]\rangle
-\bar\om(\hat X,\hat Y)(f)=({\om\cdot b\dd i_Yi_X\mu})\hat\ (f)
-({\om\cdot i_Yi_X\mu})\hat\ (f)\\
&=({\om\cdot Pi_Xi_Y\mu})\hat\ (f)
=\int_Sf^*\om\wedge Pi_Xi_Y\mu
\end{align*}
on the Lie algebra of exact divergence free vector fields,
for $P=1-b\dd$ the projection on the subspace
of closed $(m-2)$-forms.
It does not depend on $f\in\F$, because the 
cohomology class $[f^*\om]\in H^2(S)$ does not depend on
the choice of $f$. 
The cocycle $\si_\F$ is cohomologous to the Lichnerowicz cocycle
$\si_{f^*\om}$ defined in \eqref{lich} \cite{V09}.
Since $\int_S\mu=1$,
the cocycle $\si_\F$ is integrable if and only if the cohomology class of $f^*\om$ is integral \cite{I96}. 

\begin{rema}
The two equivariant momentum maps on the symplectic manifold $\F(S,M)$, 
for suitable central extensions of the hamiltonian group $\Diff_{ham}(M,\om)$ and 
of the group $\Diff_{ex}(S,\mu)$ of exact volume preserving diffeomorphisms,
form the dual pair for ideal incompressible fluid flow \cite{MW83} \cite{GBV09}.
\end{rema}

%%%%%%%%%%%%%%%%%%%%%%%%%%%%%

%%%%%%

\section{Appendix: Fiber integration}

Chapter VII in \cite{GHV72} is devoted to the concept of integration over the fiber in locally trivial bundles.
We particularize this fiber integration to the case of trivial bundles $S\x M\to M$, listing its main properties without proofs.

Let $S$ be a compact $k$--dimensional manifold. Fiber integration over $S$ 
assigns to $\om\in\Om^n(S\x M)$ the differential form $\fint_S\om\in\Om^{n-k}(M)$ defined by
\[
(\fint_S\om)(x)=\int_S\om_x\in\La^{n-k}T^*_xM,\quad\forall x\in M,
\]
where $\om_x\in\Om^k(S,\La^{n-k}T_x^*M)$ is the retrenchment of $\om$ to the fiber over $x$:
\[
\langle\om_x(Z_s^1,\dots,Z_s^{n-k}),X_x^1\wedge\dots\wedge X_x^k\rangle=\om_{(s,x)}(X_x^1,\dots,X_x^k,Z_s^1,\dots, Z_s^{n-k})
\]
for all $X_x^i\in T_xM$ and $Z_s^j\in T_sS$.

The properties of the fiber integration used in the text are
special cases of the propositions (VIII) and (X) in \cite{GHV72}:
\begin{enumerate}
\item Pull-back of fiber integrals:
\begin{equation}\label{r1}
f^*\fint_S\om=\fint_S(1_S\x f)^*\om,\quad\forall f\in\F(M',M),
\end{equation}
with infinitesimal version
\begin{equation}\label{rr}
L_X\fint_S\om=\fint_SL_{0_S\x X}\om,\quad\forall X\in\X(M).
\end{equation}
\item Invariance under pull-back by orientation preserving diffeomorphisms of $S$:
\begin{equation}\label{r2}
\fint_S(\ph\x 1_M)^*\om=\fint_S\om,\quad\forall\ph\in\Diff_+(S),
\end{equation}
with infinitesimal version 
$\fint_SL_{Z\x 0_M}\om=0,\quad\forall Z\in\X(S)$.
\item Insertion of vector fields into fiber integrals:
\begin{equation}\label{r3}
i_X\fint_S\om=\fint_Si_{0_S\x X}\om,\quad\forall X\in\X(M).
\end{equation}
%\[\int_Si_{Z\x 0_M}\om=0,\quad\forall Z\in\X(S).\]
\item Integration along boundary free manifolds commutes with differentiation. When $\pa S$ denotes the boundary of the $k$--dimensional compact manifold $S$ and $i_\pa:\pa S\to S$ the inclusion,
\begin{equation}\label{r4}
\dd\fint_S\be-\fint_S\dd\be=(-1)^{n-k}\fint_{\pa S}(i_\pa\x 1_M)^*\be
\end{equation}
holds for any differential $n$--form $\be$ on $S\x M$.
\end{enumerate}

%%%%%%%%%%%%%%%

\end{document}